\documentclass[leqno]{amsart}
\usepackage{amsmath}
\usepackage{amssymb}
\usepackage{amsthm}
\usepackage{enumerate}
\usepackage[mathscr]{eucal}
\usepackage{enumitem}
\theoremstyle{plain}
\newtheorem{theorem}{Theorem}[section]
\newtheorem{lemma}[theorem]{Lemma}
\newtheorem{prop}[theorem]{Proposition}
\theoremstyle{definition}

\newtheorem{remark}[theorem]{Remark}

\newtheorem{cor}[theorem]{Corollary}
\theoremstyle{remark}
\usepackage{amsmath, amsthm, amscd, amsfonts, amssymb, graphicx, color}
\usepackage[bookmarksnumbered, colorlinks, plainpages]{hyperref}
\hypersetup{colorlinks=true,linkcolor=red, anchorcolor=green, citecolor=cyan, urlcolor=red, filecolor=magenta, pdftoolbar=true}

\begin{document}

\title [On a new norm on $\mathcal{B}({\mathcal{H}})$  and numerical radius inequalities]{On a new norm on $\mathcal{B}({\mathcal{H}})$ and its applications to numerical radius inequalities } 

\author[ D. Sain, P. Bhunia, A. Bhanja and K. Paul]{D. Sain, P. Bhunia, A. Bhanja and K. Paul }

\address[Sain]{Department of Mathematics\\ Indian Institute of Science\\ Bengaluru 560012\\ Karnataka\\ INDIA}
\email{saindebmalya@gmail.com}

\address[Bhunia] { Department of Mathematics, Jadavpur University, Kolkata 700032, West Bengal, India}
\email{pintubhunia5206@gmail.com}

\address[Bhanja]{Department of Mathematics\\ Vivekananda College Thakurpukur\\ Kolkata\\ West Bengal\\India\\ }
\email{aniketbhanja219@gmail.com}

\address[Paul] {Department of Mathematics, Jadavpur University, Kolkata 700032, West Bengal, India}
\email{kalloldada@gmail.com}

\thanks{ Dr. Debmalya Sain feels elated to acknowledge the remarkable contribution of his beloved friend Subhajyoti Sarkar in his life. Mr. Pintu Bhunia would like to thank UGC, Govt. of India for the financial support in the form of SRF. Prof. Kallol Paul would like to thank RUSA 2.0, Jadavpur University for  partial support. }



\maketitle
\begin{abstract}

We introduce  a new norm on the space of bounded linear operators on a complex Hilbert space, which generalizes the numerical radius norm, the usual operator norm and the modified Davis-Wielandt radius. We study basic properties of this norm, including the upper and the lower bounds for it. As an application of the present study, we estimate bounds for the numerical radius of bounded linear operators. We illustrate with examples that our results improve on some of the important existing numerical radius inequalities.
\end{abstract}

\section{\textbf{Introduction}}
\noindent The purpose of the present article is to introduce a new norm, christened the $(\alpha,\beta)$-norm, on the space of bounded linear operators on a complex Hilbert space which generalizes the numerical radius norm, the usual operator norm and the recently introduced  modified Davis-Wielandt radius (see \cite{BSP}).  We study some important properties of the  $(\alpha,\beta)$-norm, and obtain upper and lower bounds for the said norm. This allows us to obtain some interesting numerical radius inequalities, that improve the existing results. Let us first introduce the following notations and terminologies.\\  

\noindent Let $\mathcal{B}({\mathcal{H}})$ denote the $C^{*}$-algebra of all bounded linear operators on a complex Hilbert space $\mathcal{H}$ with usual inner product $\langle.,.\rangle.$  For $T\in \mathcal{B}({\mathcal{H}})$, $T^*$ denotes the adjoint of $T$ and $|T|$ stands for the positive operator $(T^*T)^{\frac{1}{2}}$. The numerical range of $T$,  denoted by $W(T)$,  is  defined as the collection of all complex scalars $\langle Tx,x \rangle$ with $\|x\|=1$, i.e., $W(T)=\{\langle Tx,x \rangle:x\in \mathcal{H}, \|x\|=1\}.$ Let $\sigma (T)$ denote the spectrum of $T$. The usual operator norm, the numerical radius, the Crawford number and the spectral radius of $T$,  denoted respectively by $\|T\|, w(T), c(T)$ and $r(T)$,  are defined as follows:
\begin{eqnarray*}
\|T\|&=& \sup \{\|Tx\| : x\in \mathcal{H}, \|x\|=1\},\\
w(T)&=& \sup \{|c|: c\in W(T) \},\\
c(T)&=& \inf \{|c|: c\in W(T)\}, \\
r(T)&=& \sup \{|\lambda|: \lambda \in \sigma (T)\}.
\end{eqnarray*}
Clearly, $r(T)\leq w(T)$. Let $M_T$ and $c_T$ denote the norm attainment set and the Crawford number attainment set of $T$, respectively, i.e., 
\begin{eqnarray*}
M_T&=& \{x\in \mathcal{H}: \|Tx\|=\|T\|, \|x\|=1\},\\
c_T&=& \{x\in \mathcal{H}: c(T)=|\langle Tx,x \rangle|, \|x\|=1\}.
\end{eqnarray*}
For $T\in \mathcal{B}({\mathcal{H}})$, let $Re(T)$ and $Im(T)$ denote the real part  and the imaginary part of $T$ respectively, i.e., $Re(T)=\frac{1}{2}(T+T^*)$ and $Im(T)=\frac{1}{2i}(T-T^*)$. It is well-known that 
$$w(T)=\sup_{\theta \in \mathbb{R}}\|Re(e^{i\theta}T)\|=\sup_{\theta \in \mathbb{R}}\|Im(e^{i\theta}T)\|.$$
Also we know that $w(.)$ defines a norm on $\mathcal{B}({\mathcal{H}})$ and is equivalent to the usual operator norm, satisfying the following inequality:
\[\frac{1}{2}\|T\| \leq w(T)\leq \|T\|, ~~T\in \mathcal{B}({\mathcal{H}}).\]
The above inequality is sharp. The first inequality becomes equality if $T^2=0$ and the second inequality becomes equality if $T$ is normal. Kittaneh \cite{K03,K05} improved on the above inequality to show that 
$$w(T) \leq \frac{1}{2} \Big \| |T| + |T^*| \Big \| \leq \frac{1}{2}\|T\|+ \frac{1}{2}\|T^2\|^{\frac{1}{2}}$$ and $$\frac{1}{4}\|T^*T+TT^*\| \leq w^2(T) \leq \frac{1}{2}\|T^*T+TT^*\|.$$
Abu-Omar and Kittaneh in \cite[Th. 2.4]{OK2} proved that
 \[w^2(T)\leq \frac{1}{2}w(T^2)+\frac{1}{4}\left \|T^*T+ TT^*\right \|,\] which improves on the above upper bounds. A lot of study has been done in this direction to improve upper and lower bounds for the numerical radius, and we refer the readers to \cite{BP,BBP1,BBP,PB,SMY,Y}, and the references therein, for a comprehensive idea of the current state of the art.

\noindent Let us now introduce the $(\alpha,\beta)$-norm on $\mathcal{B}({\mathcal{H}})$. Throughout the paper we reserve $\alpha, \beta$ for non-negative real scalars, i.e., $\alpha, \beta \geq 0, $ such that $ ( \alpha, \beta ) \neq (0,0).$  Let us consider a mapping $ \|.\|_{\alpha,\beta}: \mathcal{B}({\mathcal{H}})\rightarrow \mathbb{R}^+$, defined as follows: $$\|T\|_{\alpha,\beta}=\sup \left \{\sqrt{  \alpha |\langle Tx ,x \rangle |^2+\beta \|Tx\|^2 }: x\in \mathcal{H},\|x\|=1 \right \}.$$
We observe that  $\|.\|_{\alpha,\beta}$ defines a norm on $\mathcal{B}({\mathcal{H}})$. We also observe that if $\alpha=1, \beta=0$ then $\|T\|_{\alpha,\beta}=w(T)$ and if $\alpha=0, \beta=1$ then $\|T\|_{\alpha,\beta}=\|T\|$. Moreover, if we consider $\alpha=\beta=1$ then we get the modified Davis-Wielandt radius of $T$, i.e.,  $\|T\|_{\alpha,\beta}=dw^*(T)$ (see \cite{BSP}).\\

\noindent In this paper we show that the  $(\alpha,\beta)$-norm is equivalent to the numerical radius norm and the usual operator norm. We also study the equality conditions for the said bounds. We then obtain some upper and lower bounds for the $(\alpha,\beta)$-norm  of bounded linear operators, and apply the results to obtain bounds  for the  numerical radius of bounded linear operators. Among other results obtained in this article, we prove the following three important inequalities: 
\begin{eqnarray*}
w(T) & \leq & \inf_{\alpha,\beta} \Big\{\frac{1}{\sqrt{\alpha+\beta}}\left \|\frac{\alpha}{4}(|T|+|T^*|)^2 + \beta T^*T \right \|^{\frac{1}{2}}\Big\} \leq \frac{1}{2} \Big \| |T| + |T^*| \Big \|, \\
w^2(T) & \leq  & \inf_{\alpha,\beta} \Big\{\frac{1}{\alpha+\beta}\left \| \frac{\alpha}{2} \left ( T^*T+TT^*\right)+\beta T^*T \right\| \Big\}\leq \frac{1}{2} \| T^*T + TT^* \| ~\mbox{and}\\
 w^2(T) & \leq  & \inf_{\alpha,\beta} \frac{1}{\alpha+\beta}\Big\{\frac{\alpha}{2}w(T^2)+\left \| \frac{\alpha}{4}(T^*T+TT^*)+\beta T^*T   \right \|\Big\} \\
 &\leq & \frac{1}{2}w(T^2)+\frac{1}{4} \|T^*T+TT^* \|.
\end{eqnarray*}
We discuss conditions under which the bounds obtained here is strictly sharper than the existing bounds. We also discuss conditions under which the bounds are equal. Using a similarly motivated approach, we also improve the existing lower bound for the numerical radius of an operator, as obtained in \cite[Th. 1]{K05}.

\section{ \textbf{The $(\alpha,\beta)$-norm and numerical radius inequalities}}

\noindent We begin this section with the observation that $\|.\|_{\alpha,\beta}$ defines a norm on $\mathcal{B}({\mathcal{H}})$, equivalent with the numerical radius norm and the usual operator norm.
\begin{theorem}\label{norm}
$\|.\|_{\alpha,\beta}$ defines a norm on $\mathcal{B}({\mathcal{H}})$ and is  equivalent to the numerical radius norm $ w(.)$ and the usual operator norm $\|.\|$, satisfying the following inequalities:
\[ \sqrt{{(\alpha + \beta} )} ~w(T) \leq  \|T\|_{\alpha, \beta} \leq \sqrt{{(\alpha + 4\beta})} ~w(T), \]
\[ \max \left \{\frac{\sqrt{\alpha+\beta}}{2},\sqrt{\beta}\right \}\|T\| \leq \|T\|_{\alpha,\beta} \leq \sqrt{(\alpha+\beta)}\|T\|. \]
\end{theorem}

\begin{proof}
First we prove that  $\|.\|_{\alpha,\beta}$ defines a norm on $\mathcal{B}({\mathcal{H}}).$ Clearly, it is sufficient to show that $\|.\|_{\alpha,\beta}$ satisfies the triangle inequality, as all other defining conditions for being a norm follow trivially. Let $ S,T \in \mathcal{B}({\mathcal{H}}).$  Then we get
\begin{eqnarray*}
\|S+T\|_{\alpha,\beta}^2 &=& \sup_{\|x\|=1} \left \{ \alpha |\langle (S+T)x ,x \rangle |^2+\beta \|(S+T)x\|^2 \right \} \\
&\leq & \sup_{\|x\|=1} \left \{ \alpha (|\langle Sx ,x \rangle|+|\langle Tx ,x \rangle|)^2+\beta (\|Sx\|+\|Tx\|)^2 \right \} \\
&\leq& \sup_{\|x\|=1}  \left ( \alpha |\langle Sx ,x \rangle |^2+\beta \|Sx\|^2 \right ) +\sup_{\|x\|=1}  \left ( \alpha |\langle Tx ,x \rangle |^2+\beta \|Tx\|^2 \right ) \\
&& +\sup_{\|x\|=1}\{2 ( \alpha |\langle Sx ,x \rangle | |\langle Tx ,x \rangle |+\beta \|Sx\| \|Tx\| ) \} \\
&\leq& \sup_{\|x\|=1} \left \{ \alpha |\langle Sx ,x \rangle |^2+ \beta \|Sx\|^2 \right \} + \sup_{\|x\|=1} \left \{ \alpha |\langle Tx ,x \rangle |^2+\beta \|Tx\|^2 \right \} \\
&& +2 \sup_{\|x\|=1} \sqrt{ (\alpha |\langle Sx ,x \rangle |^2+\beta \|Sx\|^2 )}\sqrt{(\alpha |\langle Tx ,x \rangle |^2+\beta \|Tx\|^2 )},\\
&=& \|S\|_{\alpha,\beta}^2+\|T\|_{\alpha,\beta}^2+2\|S\|_{\alpha,\beta}\|T\|_{\alpha,\beta} \\
&=& (\|S\|_{\alpha,\beta}+\|T\|_{\alpha,\beta})^2.
 \end{eqnarray*}
Therefore, $\|S+T\|_{\alpha,\beta} \leq \|S\|_{\alpha,\beta}+\|T\|_{\alpha,\beta}$ for all $S,T\in \mathcal{B}({\mathcal{H}}).$ Hence,
$\|.\|_{\alpha,\beta}$ defines a norm on $\mathcal{B}({\mathcal{H}}).$  Next we have,
\begin{eqnarray*}  
\|T\|^2_{\alpha,\beta} & = & \sup_{\|x\|=1} \{\alpha |\langle Tx ,x \rangle |^2+\beta \|Tx\|^2\}\\
& \leq & \sup_{\|x\|=1} \{\alpha |\langle Tx ,x \rangle |^2 \} + \sup_{\|x\|=1} \{ \beta \|Tx\|^2\}\\
 & =  & \alpha w^2(T) +\beta \|T\|^2\\
& \leq & ( \alpha + 4 \beta) w^2(T),
\end{eqnarray*}
so that $ \| T\|_{\alpha, \beta} \leq \sqrt{( \alpha + 4 \beta)} ~w(T).$  Again we have,
 \begin{eqnarray*}
\|T\|_{\alpha,\beta} &=& \sup_{\|x\|=1}\sqrt{\alpha |\langle Tx ,x \rangle |^2+\beta \|Tx\|^2}\\
&\geq &\sup_{\|x\|=1}  \sqrt{(\alpha+\beta) |\langle Tx ,x \rangle |^2}\\
&= &  \sqrt{(\alpha+\beta)}~ w(T).\\
\end{eqnarray*}
Thus we get, 
\[ \sqrt{{(\alpha + \beta} )} ~w(T) \leq  \|T\|_{\alpha, \beta} \leq \sqrt{{(\alpha + 4\beta})} ~w(T).\]
Proceeding similarly we can show that,
\[ \max \left \{\frac{\sqrt{\alpha+\beta}}{2},\sqrt{\beta}\right \}\|T\| \leq \|T\|_{\alpha,\beta} \leq \sqrt{(\alpha+\beta)}~\|T\|. \]
This completes the proof.
\end{proof}
\begin{remark}
	The classical numerical radius bounds follow easily from the above inequalities by considering either   $ \alpha =1,  \beta=0$ or $ \alpha =0, \beta = 1.$ 
\end{remark}

In the following theorems we study the equality conditions for the bounds of the $(\alpha,\beta)$-norm.  We begin with the following theorem, which characterises operators $T$ for which $ \| T \|_{\alpha, \beta} = \sqrt{\alpha w^2(T)+\beta \|T\|^2}. $
\begin{theorem}\label{equality4}
	Let $T\in \mathcal{B}({\mathcal{H}})$ and let $\alpha\beta \neq 0$. Then the following conditions are equivalent:\\
	$(i)$ $\|T\|_{\alpha,\beta}= \sqrt{\alpha w^2(T)+\beta \|T\|^2}.$\\
	$(ii)$ $T$ is normaloid, i.e, $w(T)=\|T\|$.\\
	$(iii)$ There exist a sequence of unit vectors $\{x_{n}\} $  in $ \mathcal{H}$ such that
	\[\lim_{n\to \infty}\|Tx_{n}\|=\|T\| ~and~ \lim_{n\to \infty}|\langle Tx_{n},x_{n}\rangle|=w(T).\]
\end{theorem}

\begin{proof}
	The equivalence of $(ii)$ and $(iii)$ is well-known. We only prove the equivalence of $(i)$ and $(iii)$.
	
	We first prove $(i)\Rightarrow (iii).$  Since $T\in \mathcal{B}({\mathcal{H}})$, there exists a sequence of unit vectors $\{x_n\}$ in $\mathcal{H}$ such that 
	$$\|T\|_{\alpha,\beta} = \lim_{n\to \infty} \sqrt{\alpha |\langle Tx_n ,x_n \rangle |^2+\beta \|Tx_n\|^2}.$$ 
	Clearly $ \{|\langle Tx_n ,x_n \rangle |\} $ and $\{\|Tx_n\|\}$ both are bounded sequences of real numbers, so there exists a subsequence $\{x_{n_k}\}$ of the sequence $\{x_n\}$ such that both $\{|\langle Tx_{n_k} ,x_{n_k} \rangle |\} $ and $ \{\|Tx_{n_k}\|\}$ are convergent. So we get,
	\begin{eqnarray*}
		\alpha w^2(T)+\beta \|T\|^2=\|T\|_{\alpha,\beta}^2 &=& \lim_{k\to \infty} \{\alpha |\langle Tx_{n_k} ,x_{n_k} \rangle |^2+ \beta \|Tx_{n_k}\|^2 \} \\ 
		&\leq&  \lim_{k\to \infty} \alpha |\langle Tx_{n_k} ,x_{n_k} \rangle |^2+ \lim_{k\to \infty} \beta \|Tx_{n_k}\|^2 \\
		& \leq &  \lim_{k\to \infty} \alpha |\langle Tx_{n_k} ,x_{n_k} \rangle |^2 + \beta \|T\|^2 \\
		&\leq& \alpha w^2(T)+\beta \|T\|^2.
	\end{eqnarray*}
	This implies that \[\lim_{k\to \infty}\|Tx_{n_k}\|=\|T\| ~and~ \lim_{k\to \infty}|\langle Tx_{n_k},x_{n_k}\rangle|=w(T).\]
	To show that $ (iii) \Rightarrow (i),$  we observe that
	\begin{eqnarray*}
		\|T\|_{\alpha,\beta} & = &  \sup_{ \| x \| = 1} \sqrt{  \alpha |\langle Tx ,x \rangle |^2+\beta \|Tx\|^2 }\\
		&\geq& \lim_{n\to \infty} \sqrt{\alpha |\langle Tx_n ,x_n \rangle |^2+\beta \|Tx_n\|^2}\\
		&=&\sqrt{\alpha w^2(T)+\beta \|T\|^2}.
	\end{eqnarray*}
	This completes the proof.
\end{proof}

\begin{remark}
$(i)$	In case $\mathbb{H}$ is finite-dimensional, the condition $(iii)$ in Theorem \ref{equality4} implies that there exists $ x \in \mathbb{H}, \|x\|=1$ such that $\|Tx\|= \|T\| $ and $ w(T) = $ $\mid \langle Tx, x \rangle \mid.$ \\
$(ii)$  The equality conditions for $\|T\|_{\alpha,\beta} = \sqrt{(\alpha + \beta)} w(T)$  as well as  $\|T\|_{\alpha,\beta} = \sqrt{(\alpha + \beta)} \|T\|$  follows from Theorem \ref{norm} and Theorem \ref{equality4}.
\end{remark}
We next find the equality condition for $\|T\|_{\alpha,\beta}=\sqrt{(\alpha+4 \beta)}w(T).$ 
\begin{theorem}\label{equality-w}
	Let $T\in \mathcal{B}({\mathcal{H}})$ and  let $\alpha \beta \neq 0$. Then  $\|T\|_{\alpha,\beta}=\sqrt{(\alpha+4 \beta)}w(T)$ if and only if $T=0$.
\end{theorem}
\begin{proof}
	The sufficient part is trivial. We only prove the necessary part. Let $\|T\|_{\alpha,\beta}=\sqrt{\alpha+4 \beta}~~w(T)$. Then we have,
	$(\alpha+4 \beta)~~w^2(T)=\|T\|_{\alpha,\beta}^2 \leq \alpha w^2(T)+\beta \|T\|^2 \leq \alpha w^2(T)+\beta (2w(T))^2=(\alpha+4 \beta)~~w^2(T)$. This implies that $w(T)=\frac{\|T\|}{2}.$ On the other hand using Theorem \ref{equality4}  we get,  $w(T)=\|T\|.$ Therefore, $\|T\|=0$, i.e., $T=0$. 
\end{proof}

The characterizations for the equality of other bounds can be obtained analogously, we skip them to avoid monotonicity.
Next we observe the  properties of the $(\alpha,\beta)$-norm on $\mathcal{B}({\mathcal{H}})$ in the form of  following proposition, the proof of which follows easily, except for property $(v),$  which follows from \cite{LPS}.

\begin{prop}\label{properties}
	The $(\alpha,\beta)$-norm on $\mathcal{B}({\mathcal{H}})$ satisfies the following properties:\\
	
	\noindent $(i)$  $\|.\|_{\alpha,\beta}$ is not an algebra norm, i.e., there exists $A, B \in \mathcal{B}({\mathcal{H}})$  for which  $$\|AB\|_{\alpha,\beta}\leq \|A\|_{\alpha,\beta}~~\|B\|_{\alpha,\beta}$$ 
	does not hold.

	\noindent $(ii)$  $\|.\|_{\alpha,\beta}$ does not satisfy the power inequality, i.e., there exist operators $A,B\in \mathcal{B}({\mathcal{H}})$ and  $n\in \mathbb{N}\setminus \{1\}$ such that 
	$$\|A^n\|_{\alpha,\beta}<\|A\|^n_{\alpha,\beta} ~ \mbox{and} ~ \|B^n\|_{\alpha,\beta}>\|B\|^n_{\alpha,\beta}.$$
	
	\noindent $(iii) $ If $\alpha+\beta=1$ and $A$ is normal then $\|A^n\|_{\alpha,\beta}=\|A\|^n_{\alpha,\beta},$ for all $n\in \mathbb{N}.$\\
	
	\noindent $(iv)$ $\|.\|_{\alpha,\beta}$ is weakly unitarily invariant , i.e.,  
	$$\|U^*TU\|_{\alpha,\beta}=\|T\|_{\alpha,\beta}, ~ \forall ~ T\in \mathcal{B}({\mathcal{H}}),$$ where $U \in \mathcal{B}({\mathcal{H}})$  is an unitary operator.\\
	
	\noindent $(v)$ $\|.\|_{\alpha,\beta}$ is preserved under  the adjoint operation, i.e., $$\|T\|_{\alpha,\beta}=\|T^*\|_{\alpha,\beta}, \forall ~ T\in \mathcal{B}({\mathcal{H}}).$$
\end{prop}

We next obtain the following lower bound for the $(\alpha,\beta)$-norm.

\begin{theorem}\label{est2a}
Let $T \in \mathcal{B}({\mathcal{H}})$. Then 
\begin{eqnarray*}
 \|T\|_{\alpha,\beta}^2 & \geq & \max ~\Big \{ \alpha w^2(T)+\beta c(T^*T), ~\alpha c^2(T)+\beta \|T\|^2,\\
                                                    &           &  2\sqrt{\alpha \beta}w(T)\sqrt{c(T^*T)}, ~2\sqrt{\alpha \beta}c(T)\|T\|\Big  \}.
\end{eqnarray*}
\end{theorem}

\begin{proof}
For all $x \in \mathcal{H}$ with $\|x\|=1,$ we get,
\begin{eqnarray*}
\|T\|_{\alpha,\beta}^2 &\geq& \alpha |\langle Tx,x \rangle|^2+ \beta \|Tx\|^2 \\
&=& \alpha |\langle Tx,x \rangle|^2+\beta \langle T^*Tx,x \rangle \\
&\geq& \alpha |\langle Tx,x \rangle|^2+ \beta c(T^*T).
\end{eqnarray*}
Therefore, taking supremum over all unit vectors in $\mathcal{H},$ we get,
\[\|T\|_{\alpha,\beta}^2 \geq \alpha w^2(T)+\beta c(T^*T).\]
Also, $\|T\|_{\alpha,\beta}^2 \geq \alpha |\langle Tx,x \rangle|^2+ \beta \|Tx\|^2 \geq \alpha c^2(T)+\beta \|Tx\|^2.$
Taking supremum over all unit vectors in $\mathcal{H},$ we get, $\|T\|_{\alpha,\beta}^2 \geq \alpha c^2(T)+\beta \|T\|^2$.

\noindent Again, we have
 \begin{eqnarray*}
\|T\|_{\alpha,\beta}^2 &\geq& \alpha |\langle Tx,x \rangle|^2+ \beta \|Tx\|^2 \\
&\geq& 2\sqrt{\alpha \beta}|\langle Tx,x \rangle|\|Tx\|\\
&\geq& 2\sqrt{\alpha \beta}|\langle Tx,x \rangle|\sqrt{c(T^*T)}.
\end{eqnarray*}
Therefore,
$$\|T\|_{\alpha,\beta}^2 \geq 2\sqrt{\alpha \beta}w(T)\sqrt{c(T^*T)}.$$
We also observe that $\|T\|_{\alpha,\beta}^2 \geq 2\sqrt{\alpha \beta}|\langle Tx,x \rangle|\|Tx\| \geq 2\sqrt{\alpha \beta} c(T)\|Tx\|.$
 Therefore, $\|T\|_{\alpha,\beta}^2 \geq 2\sqrt{\alpha \beta} c(T)\|T\|.$\\
Combining the above inequalities we get the required inequality.
\end{proof}

\noindent We next obtain bounds for the  $(\alpha,\beta)$-norm of the product of two bounded linear operators. We require the following known lemmas, which can be found in \cite[pp. 37-39]{GR}.
\begin{lemma}\label{lem1}
Let $A,B \in \mathcal{B}({\mathcal{H}}).$ Then \\
$(i)$ $w(AB)\leq 4w(A)w(B)$.\\
$(ii)$  If $AB=BA$ then $w(AB)\leq 2w(A)w(B).$ \\
$(iii) $  If $AB=BA$ and $A$ is an isometry, then $w(AB)\leq w(B).$ \\
$(iv)$ If $AB=BA$ and $AB^*=B^*A,$ then $w(AB)\leq w(B)\|A\|.$
\end{lemma}

\noindent We are now in a position to obtain upper bounds for the $(\alpha,\beta)$-norm of the product of two bounded linear operators.

\begin{theorem}
Let $A,B \in \mathcal{B}({\mathcal{H}})$ and let $\beta \neq 0$. Then we have the following inequality: %
\begin{eqnarray*}
\|AB\|_{\alpha,\beta} &\leq&  \sqrt{\min \left\{\frac{4}{\beta}, \frac{\alpha+\beta}{\beta^2}, \frac{16}{\alpha+\beta} \right \}}\|A\|_{\alpha,\beta}\|B\|_{\alpha,\beta}.
\end{eqnarray*}
\end{theorem}

\begin{proof}
From the definition of $(\alpha,\beta)$-norm, we have
\begin{eqnarray*}
\|AB\|_{\alpha,\beta}^2 &=& \sup_{\|x\|=1}\left \{ \alpha |\langle ABx ,x \rangle |^2+\beta \|ABx\|^2 \right \}\\ 
&\leq& \alpha w^2(AB)+\beta \|AB\|^2\\
&\leq& (\alpha+\beta)\|A\|^2\|B\|^2\\
& \leq & 4(\alpha+\beta)w^2(A)\|B\|^2, ~\mbox{ using } ~ \|A\| \leq 2 w(A)\\
& \leq & 4(\alpha+\beta)w^2(A) ( 1/\beta) \|B\|_{\alpha, \beta}^2, \\
& \leq & \frac{4}{\beta}\|A\|_{\alpha,\beta}^2\|B\|_{\alpha,\beta}^2, ~ \mbox{using Theorem  \ref{norm}}.
\end{eqnarray*}
Thus we get, $$ \|AB\|_{\alpha,\beta}\leq \sqrt{ \frac{4}{\beta}} \|A\|_{\alpha,\beta}\|B\|_{\alpha,\beta}.$$
Proceeding as above and using $ \sqrt{\beta}\|A\| \leq \|A\|_{\alpha, \beta}, ~\sqrt{\beta}\|B\| \leq \|B\|_{\alpha, \beta} ,$ we get,
$$  \|AB\|_{\alpha,\beta} \leq \sqrt{ \frac{\alpha + \beta}{\beta^2}} \|A\|_{\alpha,\beta}\|B\|_{\alpha,\beta}.$$ 
Finally using  $  \|A\| \leq 2 w(A), ~ \|B\| \leq 2 w(B) $ and Theorem  ~\ref{norm} we obtain, 

 $$ \|AB\|_{\alpha,\beta} \leq \sqrt{\frac{16}{\alpha+\beta} } \|A\|_{\alpha,\beta}\|B\|_{\alpha,\beta}. $$
 Combining all the above inequalities we get the required inequality.

\end{proof}

In the following theorem we obtain upper bounds for the $(\alpha,\beta)$-norm of the product of two bounded linear operators, under the additional assumption that they commute. 
\begin{theorem}
Let $A,B \in \mathcal{B}({\mathcal{H}})$ be such that $AB=BA$.  \\
$ (i)$ If $ \beta \neq 0$, then 
$$ \|AB\|_{\alpha,\beta} \leq \sqrt{\left(\frac{4\alpha}{(\alpha+\beta)^2}+\frac{1}{\beta} \right )}\|A\|_{\alpha,\beta}\|B\|_{\alpha,\beta}. $$
$(ii)$  If $A$ is an isometry, then
 $$	\|AB\|_{\alpha,\beta} \leq\sqrt{\left ( \frac{\alpha}{\alpha+\beta}+1  \right )}\|B\|_{\alpha,\beta}. $$
$(iii)$  If  $AB^*=B^*A$ and  $\beta \neq 0$, then  
\[ \|AB\|_{\alpha,\beta}  \leq \sqrt{\left (\frac{\alpha}{\alpha+\beta}+1\right )} \min \left \{\frac{2}{\sqrt{\alpha+\beta}},\frac{1}{\sqrt{\beta}}\right \}\|A\|_{\alpha,\beta}\|B\|_{\alpha,\beta}.\]
\end{theorem}

\begin{proof} $(i).$  
Using the definition of the $(\alpha,\beta)$ norm and Lemma \ref{lem1} (ii), we get,
\begin{eqnarray*}
\|AB\|_{\alpha,\beta}^2 &=& \sup_{\|x\|=1}\left \{ \alpha |\langle ABx ,x \rangle |^2+\beta \|ABx\|^2 \right \}\\ 
&\leq& \alpha w^2(AB)+\beta\|AB\|^2\\
&\leq& 4\alpha w^2(A)w^2(B)+\beta\|A\|^2\|B\|^2\\
&\leq& \left( \frac{4\alpha}{(\alpha+\beta)^2}+ \frac{1}{\beta} \right )\|A\|_{\alpha,\beta}^2\|B\|_{\alpha,\beta}^2, ~~\mbox{using Theorem  ~\ref{norm}.}
\end{eqnarray*}
Thus we get the inequality in $(i).$ \\
$(ii)$.  As above, using the definition of $(\alpha,\beta)$ norm and Lemma \ref{lem1} (iii), we get,
\begin{eqnarray*}
	\|AB\|_{\alpha,\beta}^2 &=& \sup_{\|x\|=1}\left \{ \alpha |\langle ABx ,x \rangle |^2+\beta \|ABx\|^2 \right \}\\ 
	&\leq& \alpha w^2(AB)+\beta\|AB\|^2\\
	&\leq& \alpha w^2(B)+\beta\|B\|^2, \\
	&\leq& \left ( \frac{\alpha}{\alpha+\beta}+1  \right ) \|B\|_{\alpha,\beta}^2, ~~\mbox{using Theorem  ~\ref{norm}.}
\end{eqnarray*}
Thus we get the inequality in $(ii).$ \\
$(iii)$. Proceeding as in the above two cases and using Lemma \ref{lem1} (iv), we get the required inequality.

\end{proof}
To proceed further in the estimation of upper bound for the $(\alpha,\beta)$-norm of  product of  two bounded linear operators, we need the following two lemmas.
\begin{lemma}$($\cite{SMY}$)$\label{lem4}
$(i)$ The Power-Mean inequality:
$$ a^tb^{(1-t)} ~\leq~ ta +(1-t)b ~\leq ~\left (ta^p +(1-t)b^p \right)^\frac{1}{p}, $$
for all $t\in [0,1]$, $a,b \geq 0$ and $p \geq 1.$\\

$(ii)$ The Power-Young inequality:
$$ ab ~\leq~  \frac{a^n}{n}+\frac{b^m}{m} ~\leq~  \left (  \frac{a^{pn}}{n}+\frac{b^{pm}}{m} \right)^\frac{1}{p},$$
for all $a,b \geq 0$ and $n,m >1$ with $\frac{1}{n}+\frac{1}{m}=1$ and $p\geq 1.$
\end{lemma}

\begin{lemma}$($\cite{K88}$)$\label{lem5}
 Let $A \in \mathcal{B}({\mathcal{H}})$ be a positive operator i.e., $A \geq 0.$ Then for any unit vector $x \in \mathcal{H}$, we have the following inequality:
\[\langle Ax,x \rangle ^p \leq \langle A^px,x \rangle ,\]  
for all $p\geq 1$.  
\end{lemma}

\begin{theorem}\label{est2}
Let $A,B \in \mathcal{B}(\mathcal{H})$ be such that $AB=BA$ and $A^*B=BA^*.$ Also let $ \alpha + \beta = 1.$  Then
\begin{eqnarray*}
\|AB\|_{\alpha,\beta} &\leq& \left (2 w\left(\begin{array}{cc}
	0 & \frac{1}{n}\left ( \alpha (AA^*)^{pn}+\beta (A^*A)^{pn} \right)\\
	\frac{1}{m}(B^*B)^{pm}  & 0
	\end{array}\right)\right )^{\frac{1}{2p}}, \\
\end{eqnarray*}
where $n,m >1$ with $\frac{1}{n}+\frac{1}{m}=1,$ $p\geq 1,$ $pn \geq2,$ $pm \geq2.$ 
\end{theorem}

\begin{proof}
Let $x$ be a unit vector in $\mathcal{H}.$ Then by the convexity of the function $t^p$, we get
\begin{eqnarray*}
 \left (\alpha|\langle ABx, x\rangle |^2  + \beta \| ABx \|^2\right)^p &\leq&  \alpha |\langle ABx, x\rangle |^{2p}  + \beta \| ABx \|^{2p}.
\end{eqnarray*}
Using Lemma \ref{lem4} (ii) and Lemma \ref{lem5}, we get
\begin{eqnarray*}
\alpha |\langle ABx, x\rangle |^{2p}  + \beta \| ABx \|^{2p} &=& \alpha |\langle Bx,A^*x   \rangle|^{2p}  +\beta \langle A^*Ax,B^*Bx   \rangle ^p \\ 
&\leq&  \alpha\|Bx\|^{2p}\|A^*x\|^{2p} + \beta \|A^*Ax\|^{p}\|B^*Bx\|^{p}\\
&=&  \alpha \langle AA^*x,x   \rangle  ^p\langle B^*Bx,x \rangle  ^p+\beta \langle  (A^*A)^2x,x  \rangle  ^{\frac{p}{2}}\langle (B^*B)^2x,x \rangle  ^{{\frac{p}{2}}}\\
&\leq&  \alpha\left (\frac{1}{n} \langle AA^*x,x  \rangle^{pn} +\frac{1}{m} \langle B^*Bx,x \rangle ^{pm} \right )\\
&& +  \beta  \left ( \frac{1}{n} \langle (A^*A)^2x,x \rangle^{\frac{pn}{2}} +\frac{1}{m} \langle (B^*B)^2x,x \rangle ^{\frac{pm}{2}} \right ) \\
&\leq& \alpha\left (\frac{1}{n} \langle (AA^*)^{pn}x,x  \rangle +\frac{1}{m} \langle (B^*B)^{pm}x,x \rangle  \right )\\
&& +  \beta  \left ( \frac{1}{n} \langle (A^*A)^{pn}x,x \rangle +\frac{1}{m} \langle (B^*B)^{pm}x,x \rangle  \right ) \\
&=& \left \langle \left (\frac{1}{n}\left ( \alpha (AA^*)^{pn}+\beta (A^*A)^{pn} \right)+ \frac{1}{m}(B^*B)^{pm}     \right)x,x \right \rangle \\
&\leq& \left \| \frac{1}{n}\left ( \alpha (AA^*)^{pn}+\beta (A^*A)^{pn} \right)+ \frac{1}{m}(B^*B)^{pm}    \right \| \\
&=& 2 w\left(\begin{array}{cc}
	0 & \frac{1}{n}\left ( \alpha (AA^*)^{pn}+\beta (A^*A)^{pn} \right)\\
	\frac{1}{m}(B^*B)^{pm}  & 0
	\end{array}\right). 
\end{eqnarray*}
Therefore, 
\begin{eqnarray*}
 \left (\alpha|\langle ABx, x\rangle |^2  + \beta \| ABx \|^2\right)^p &\leq& 2 w\left(\begin{array}{cc}
	0 & \frac{1}{n}\left ( \alpha (AA^*)^{pn}+\beta (A^*A)^{pn} \right)\\
	\frac{1}{m}(B^*B)^{pm}  & 0
	\end{array}\right).
\end{eqnarray*}
Taking supremum over all unit vectors in $\mathcal{H}$, we get the required inequality.
\end{proof}

Next we obtain an inequality for the $(\alpha,\beta)$-norm of product of two bounded linear operators in terms of non-negative continuous functions on $[0,\infty)$. We  need the following lemma.

\begin{lemma}\cite[Th. 5]{K88}\label{lem7}
Let $A,B \in \mathcal{B}({\mathcal{H}})$ be such that $|A|B=B^*|A|.$  If $f$ and $g$ are two non-negative continuous functions on $[0,\infty)$ satisfying $f(t)g(t)=t, ~\forall t\geq 0$ then 
\[ |\langle ABx,y \rangle | \leq r(B)\|f(|A|)x\|\|g(|A^*|)y\|, \]
for any vectors $x,y\in \mathcal{H}.$
\end{lemma}

\begin{theorem}\label{est3a}
Let $A,B \in \mathcal{B}({\mathcal{H}})$ be such that $|A|B=B^*|A|$ and let $ \alpha + \beta = 1.$ If $f$ and $g$ are two non-negative continuous functions on $[0,\infty)$ satisfying $f(t)g(t)=t $, $\forall ~ t\geq 0$ then\\
$	(i) ~\|AB\|_{\alpha,\beta} \leq \left \| \alpha r^{2p}(B)\left ( \frac{1}{n} f^{np}(|A|) + \frac{1}{m} g^{mp}(|A^*|)\right )^2 +\beta ((AB)^*(AB))^p  \right \|^{\frac{1}{2p}} ~ \mbox{and} $ \\
	$ (ii)~	w(AB) \leq \inf_{\alpha+ \beta=1 } \left \| \alpha r^{2p}(B)\left ( \frac{1}{n} f^{np}(|A|) + \frac{1}{m} g^{mp}(|A^*|)\right )^2 +\beta ((AB)^*(AB))^p  \right \|^{\frac{1}{2p}}, $\\
where $n,m > 1$ with $\frac{1}{n}+\frac{1}{m}=1,$ $p\geq 1,$ $pn \geq2$ and $pm \geq2.$

\end{theorem}

\begin{proof}
Let $x$ be a unit vector in $\mathcal{H}.$ Then using Lemma \ref{lem7} and Lemma \ref{lem4} (ii), we get
\begin{eqnarray*}
|\langle ABx,x \rangle| &\leq& r(B)\|f(|A|)x\|\|g(|A^*|)x\|\\
 \Rightarrow  |\langle ABx,x \rangle| &\leq&  r(B)\left (\frac{1}{n}\| f(|A|)x\|^n +\frac{1}{m}\|g(|A^*|)x \|^m  \right )\\
 \Rightarrow  |\langle ABx,x \rangle| &\leq & r(B)\left (\frac{1}{n} \left \langle   f^2(|A|)x ,x  \right \rangle ^\frac{n}{2}  +\frac{1}{m}  \left \langle  g^2(|A^*|)x,x   \right \rangle ^\frac{m}{2} \right )\\
\Rightarrow  |\langle ABx,x \rangle|^p &\leq& r^p(B)\left (\frac{1}{n} \left \langle   f^2(|A|)x ,x  \right \rangle ^\frac{n}{2}  +\frac{1}{m}  \left \langle  g^2(|A^*|)x,x   \right \rangle ^\frac{m}{2} \right )^p\\
\Rightarrow  |\langle ABx,x \rangle|^p &\leq& r^p(B)\left (\frac{1}{n} \left \langle   f^2(|A|)x ,x  \right \rangle ^\frac{np}{2}  +\frac{1}{m}  \left \langle  g^2(|A^*|)x,x   \right \rangle ^\frac{mp}{2} \right )\\
\Rightarrow  |\langle ABx,x \rangle|^p &\leq&  r^p(B)\left (\frac{1}{n} \left \langle   f^{np}(|A|)x ,x  \right \rangle  +\frac{1}{m}  \left \langle  g^{mp}(|A^*|)x,x   \right \rangle \right )\\ 
\Rightarrow  |\langle ABx,x \rangle|^{2p} &\leq& r^{2p}(B)\left \langle \left ( \frac{1}{n} f^{np}(|A|) + \frac{1}{m} g^{mp}(|A^*|)\right )x,x    \right \rangle^2\\
\Rightarrow  |\langle ABx,x \rangle|^{2p} &\leq& r^{2p}(B)\left \langle \left ( \frac{1}{n} f^{np}(|A|) + \frac{1}{m} g^{mp}(|A^*|)\right )^2x,x    \right \rangle.\\
\end{eqnarray*}

Now using Lemma \ref{lem5}, we get
$$\|ABx\|^{2p} = \langle ABx,ABx \rangle ^p = \langle (AB)^*(AB)x,x \rangle ^p \leq \langle ((AB)^*(AB))^px,x \rangle .$$
Then  using Lemma \ref{lem4} (ii), we get
\begin{eqnarray*}
\|AB\|_{\alpha,\beta}^{2p} &=&  \sup_{\|x\|=1}\left \{ \alpha |\langle ABx ,x \rangle |^2+\beta \|ABx\|^2 \right \}^p\\
&\leq&  \sup_{\|x\|=1}\left \{ \alpha |\langle ABx ,x \rangle |^{2p}+\beta \|ABx\|^{2p} \right \}\\
&\leq&  \sup_{\|x\|=1}\left \langle \left ( \alpha r^{2p}(B)\left ( \frac{1}{n} f^{np}(|A|) + \frac{1}{m} g^{mp}(|A^*|)\right )^2 +\beta ((AB)^*(AB))^p \right )x,x    \right \rangle \\
&=& \left \| \alpha r^{2p}(B)\left ( \frac{1}{n} f^{np}(|A|) + \frac{1}{m} g^{mp}(|A^*|)\right )^2 +\beta ((AB)^*(AB))^p  \right \|.
\end{eqnarray*}
Thus we obtain the inequality in $(i)$.
Next using Theorem \ref{norm}, namely, $ w(AB) \leq  \frac{1}{\sqrt{\alpha + \beta} } \|AB\|_{\alpha,\beta},$ we get,
\begin{eqnarray*}
	w(AB) & \leq &   \frac{1}{\sqrt{\alpha+\beta}} \left \| \alpha r^{2p}(B)\left ( \frac{1}{n} f^{np}(|A|) + \frac{1}{m} g^{mp}(|A^*|)\right )^2 +\beta ((AB)^*(AB))^p  \right \|^{\frac{1}{2p}}. 
\end{eqnarray*}
Taking infimum over $\alpha, \beta,$ with $\alpha+ \beta=1$ we get the  inequality in $(ii)$.
\end{proof}

\begin{remark}
It is clear that the inequality obtained in Theorem \ref{est3a}(ii) improves on the existing inequalities \cite[Th. 2.5]{BP} and \cite[Th. 3.1]{A}.
\end{remark}
Next we prove the following inequality:

\begin{theorem}\label{est3}
Let $A,B \in \mathcal{B}({\mathcal{H}})$ be such that $AB=BA,AB^*=B^*A$ and $|A|B=B^*|A|$ and let $\alpha+\beta=1.$  If $f$ and $g$ are two non-negative continuous functions on $[0,\infty)$ satisfying $f(t)g(t)=t $, $\forall ~ t\geq 0$,  then  
\begin{eqnarray*}
(i)~~\|AB\|_{\alpha,\beta} &\leq& \left \| \alpha r^{2p}(B) X^2 +\beta r^p(B^*B) Y \right \|^{\frac{1}{2p}}~~\textit{and}\\
(ii)~~w(AB) &\leq&  \inf_{\alpha+ \beta=1} \left \| \alpha r^{2p}(B) X^2 +\beta r^p(B^*B) Y \right \|^{\frac{1}{2p}},
\end{eqnarray*}
where  $X=\frac{1}{n} f^{np}(|A|) + \frac{1}{m} g^{mp}(|A^*|), ~Y=\frac{1}{n} f^{np}(|A^*A|)+ \frac{1}{m}g^{mp}(|A^*A|), $ and  $n,m > 1$ with $\frac{1}{n}+\frac{1}{m}=1,$ $p\geq 1,$ $pn \geq2,$ $pm \geq2.$
\end{theorem}

\begin{proof}
Let $x \in \mathcal{H}$ with $\|x\|=1.$ 
Proceeding similarly as in the proof of Theorem \ref{est3a}, we get 
$$|\langle ABx,x \rangle|^{2p} \leq  r^{2p}(B)\left \langle \left ( \frac{1}{n} f^{np}(|A|) + \frac{1}{m} g^{mp}(|A^*|)\right )^2x,x    \right \rangle.$$
\noindent Noting that $(AB)^*(AB)=A^*AB^*B, ~|A^*A|B^*B=B^*B|A^*A|$ and using Lemma ~\ref{lem7} we get,
\begin{eqnarray*}
\|ABx\|^2 &=& \langle A^*AB^*Bx,x \rangle \\
\Rightarrow \|ABx\|^2 &\leq& r(B^*B) \left\|f(|A^*A|)x  \right \|\left\| g(|A^*A|)x \right \|\\
\Rightarrow \|ABx\|^{2p} &\leq& r^p(B^*B) \left (\frac{1}{n}\left\|f(|A^*A|)x  \right \|^n   + \frac{1}{m} \left\| g(|A^*A|)x \right \|^m  \right )^p\\
\Rightarrow \|ABx\|^{2p} &\leq& r^p(B^*B) \left (\frac{1}{n} \left \langle f^2(|A^*A|)x ,x \right \rangle ^ \frac{np}{2} +\frac{1}{m} \left \langle g^2(|A^*A|)x ,x \right \rangle ^\frac{mp}{2}  \right )\\
\Rightarrow \|ABx\|^{2p} &\leq& r^p(B^*B) \left (\frac{1}{n} \left \langle f^{np}(|A^*A|)x ,x \right \rangle  +\frac{1}{m} \left \langle g^{mp}(|A^*A|)x ,x \right \rangle   \right )\\
\Rightarrow \|ABx\|^{2p} &\leq& r^p(B^*B) \left \langle \left ( \frac{1}{n} f^{np}(|A^*A|)+ \frac{1}{m}g^{mp}(|A^*A|)\right )x , x \right \rangle.\\
\end{eqnarray*}

Next using Lemma \ref{lem4}, we get

\begin{eqnarray*}
\|AB\|_{\alpha,\beta}^{2p} &=&  \sup_{\|x\|=1}\left \{ \alpha |\langle ABx ,x \rangle |^2+\beta \|ABx\|^2 \right \}^p\\
&\leq&  \sup_{\|x\|=1}\left \{ \alpha |\langle ABx ,x \rangle |^{2p}+\beta \|ABx\|^{2p} \right \}\\
&\leq& \sup_{\|x\|=1}  \left \{ \left \langle \left (\alpha r^{2p}(B) X^2 +\beta r^p(B^*B)Y \right )x , x \right \rangle \right \} \\
&=& \left \| \alpha r^{2p}(B)X^2 +\beta r^p(B^*B)Y \right \|,
\end{eqnarray*}
where  $X=\frac{1}{n} f^{np}(|A|) + \frac{1}{m} g^{mp}(|A^*|)  $ and $ Y=\frac{1}{n} f^{np}(|A^*A|)+ \frac{1}{m}g^{mp}(|A^*A|).$ 
Thus we obtain the inequality in (i). Next using Theorem \ref{norm}, namely, $ w(AB) \leq  \frac{1}{\sqrt{\alpha + \beta} } \|AB\|_{\alpha,\beta},$ we get
\begin{eqnarray*}
w(AB) &\leq&  \frac{1}{\sqrt{\alpha+\beta}}\left \| \alpha r^{2p}(B) X^2 +\beta r^p(B^*B) Y \right \|^{\frac{1}{2p}}.
\end{eqnarray*}
Taking infimum over $\alpha, \beta$, with $\alpha+ \beta=1,$ we get the  inequality in $(ii)$.
\end{proof}

In particular, if we consider $n=m=2,$ $p=1$ and $f(t)=g(t)=t^{\frac{1}{2}}$ in Theorem \ref{est3}(i), then we get the following inequality:

\begin{cor}\label{cor-k-03}
Let $A,B \in \mathcal{B}({\mathcal{H}})$ such that $AB=BA,AB^*=B^*A$ and $|A|B=B^*|A|.$ Also let $ \alpha + \beta = 1.$ Then
\begin{eqnarray*}
\|AB\|_{\alpha,\beta}^2 &\leq& \left \| \frac{\alpha}{4} r^2(B) \left (|A|+|A^*| \right)^2 +\beta r(B^*B)A^*A \right \|.
\end{eqnarray*}
\end{cor}

In the  following  two theorems we obtain inequalities for the $(\alpha,\beta)$-norm of a bounded linear operator which generalize and improve some existing numerical radius inequalities.

\begin{theorem}\label{est4}
Let $T \in \mathcal{B}(\mathcal{H}).$ If $f$ and $g$ are two non-negative continuous functions on $[0,\infty)$ satisfying $f(t)g(t)=t,~\forall ~ t\geq 0$ then 
\begin{eqnarray*}
(i)~~\|T\|_{\alpha,\beta}^2 &\leq& \left \| \frac{\alpha }{4} \left ( f^2(|T|)+  g^2(|T^*|) \right) ^2+\beta T^*T \right\|~~\textit{and}\\
(ii)~~w(T) &\leq& \inf_{\alpha,\beta} \frac{1}{\sqrt{\alpha+\beta}}\left \|\frac{\alpha}{4}(|T|+|T^*|)^2 + \beta T^*T \right \|^{\frac{1}{2}} 
 \leq \frac{1}{2} \left \| |T|+|T^*|  \right \|
\end{eqnarray*}
\end{theorem}

\begin{proof}
Let $x \in \mathcal{H}$ with $\|x\|=1$. Then using Lemma \ref{lem7} and Lemma \ref{lem5}, we get
\begin{eqnarray*}
|\langle Tx,x \rangle| &\leq& \|f(|T|)x\|\|g(|T^*|)x\|\\
\Rightarrow  |\langle Tx,x \rangle| &\leq&  \langle f^2(|T|)x,x \rangle^{\frac{1}{2}}  \langle g^2(|T^*|)x ,x \rangle ^{\frac{1}{2}}\\
\Rightarrow  |\langle Tx,x \rangle| &\leq&  \frac{1}{2} \left ( \langle f^2(|T|)x,x \rangle+ \langle g^2(|T^*|)x ,x \rangle    \right) \\
\Rightarrow  |\langle Tx,x \rangle|^2 &\leq&    \left \langle \frac{1}{2}  \left (f^2(|T|)+  g^2(|T^*|)\right)x ,x \right \rangle ^2  \\
\Rightarrow  |\langle Tx,x \rangle|^2 &\leq&     \left \langle \frac{1}{4}  \left (f^2(|T|)+ g^2(|T^*|)\right)^2x ,x  \right \rangle. 
\end{eqnarray*}
 
Next, from the definition of the $(\alpha,\beta)$-norm, we get,
\begin{eqnarray*}
\|T\|_{\alpha,\beta}^2 &=& \sup_{\|x\|=1} \left \{ \alpha |\langle Tx,x \rangle|^2+ \beta \|Tx\|^2  \right\}\\
&\leq& \sup_{\|x\|=1} \left \{ \alpha\left \langle \frac{1}{4}  \left (f^2(|T|)+  g^2(|T^*|)\right)^2x ,x  \right \rangle +\beta \langle T^*Tx,x \rangle    \right\}\\
&=& \sup_{\|x\|=1} \left \langle \left ( \frac{\alpha}{4}  \left (f^2(|T|)+  g^2(|T^*|)\right)^2 + \beta T^*T\right )x,x \right \rangle \\
&=& \left \|   \frac{\alpha}{4}  \left (f^2(|T|)+  g^2(|T^*|)\right)^2 + \beta T^*T  \right \|.
\end{eqnarray*}
Thus we obtain the inequality in (i). In particular, if we take $f(t)=g(t)=\sqrt{t}$ in (i) and apply Theorem \ref{norm}, namely, $ w(T) \leq  \frac{1}{\sqrt{\alpha + \beta} } \|T\|_{\alpha,\beta},$ we get
\begin{eqnarray*}
w(T) &\leq&  \frac{1}{\sqrt{\alpha+\beta}}\left \|\frac{\alpha}{4}(|T|+|T^*|)^2 + \beta T^*T \right \|^{\frac{1}{2}}.
\end{eqnarray*}
Taking infimum over all $\alpha, \beta,$ we get the  inequality 
$$ w(T) \leq \inf_{\alpha,\beta} \Big \{ \frac{1}{\sqrt{\alpha+\beta}}\left \|\frac{\alpha}{4}(|T|+|T^*|)^2 + \beta T^*T \right \|^{\frac{1}{2}}\Big\}.$$
The remaining inequality follows from the case $ \alpha =1, \beta = 0 .$ 
\end{proof}

In the following theorem, we provide a condition on the operator $T,$ for which the bound obtained  in  Theorem \ref{est4}(ii)  is strictly sharper   than the inequality given in \cite[Th. 1]{K03}.

\begin{theorem}\label{less2}
Let $T\in \mathcal{B}(\mathcal{H}).$  \\
$(i) $  Suppose there exists $\alpha,\beta$ with $\alpha\beta\neq 0$ such that  $ M_{\alpha A+\beta B}\neq \emptyset,$ where $A=\frac{1}{4}(|T|+|T^*|)^2$ and $B=T^*T.$  Let $L=\mbox{span} \{M_{\alpha A+\beta B}\}.$ If $\left \| B\right\|_L<\|A\|$,   then $$w(T) \leq \frac{1}{\sqrt{\alpha+\beta}}\left \|\frac{\alpha}{4}(|T|+|T^*|)^2 + \beta T^*T \right \|^{\frac{1}{2}}<\frac{1}{2}\|~~ |T|+|T^* |~~\|.$$
$(ii)$  If  $4|T|^2-(|T|+|T^*|)^2 < 0$ and $ \beta \neq 0,$  then 
$$w(T) \leq \frac{1}{\sqrt{\alpha+\beta}}\left \|\frac{\alpha}{4}(|T|+|T^*|)^2 + \beta T^*T \right \|^{\frac{1}{2}}<\frac{1}{2}\|~~ |T|+|T^* |~~\|.$$
\end{theorem}

\begin{proof}
$(i)$ 	Suppose there exist $\alpha,\beta $ with $ \alpha\beta\neq 0$ such that $\left \| B\right\|_L<\|A\|$. Let $x\in M_{\alpha A+\beta B}.$ Then we have, 
	\begin{eqnarray*} 
		w^2(T) \leq \frac{1}{\alpha+\beta} \left \|\alpha A+\beta B\right \| &=& \frac{1}{\alpha+\beta} \left \|(\alpha A+\beta B\right)x \|\\
		&\leq & \frac{\alpha}{\alpha+\beta}\|Ax\|+\frac{\beta}{\alpha+\beta}\|Bx\|\\
		&< & \frac{\alpha}{\alpha+\beta}\|A\|+\frac{\beta}{\alpha+\beta}\|A\|\\
		&=& \|A\|.
	\end{eqnarray*}
  \begin{eqnarray*} (ii)~~~~w^2(T) & \leq &  \frac{1}{\alpha+\beta}\left \|\frac{\alpha}{4}(|T|+|T^*|)^2 + \beta T^*T \right \|\\
	& = &  \left \| \frac{1}{4}\left ( |T|+|T^*|\right)^2+\frac{\beta}{\alpha+\beta}\left(|T|^2-\frac{1}{4}(|T|+|T^*|)^2\right) \right\|\\
	& < & \frac{1}{4} \left \|~~ |T|+|T^*|~~\right \|^2
\end{eqnarray*}
for all $\alpha, \beta$ with $\beta\neq 0$. 
	This completes the proof of the theorem.
\end{proof}

\begin{remark}
  We give an example to show that the condition mentioned in Theorem \ref{less2} is not necessary.  Consider $T=\left(\begin{array}{ccc}
	0 & 0 & 0\\
	2 & 0 & 0\\
	0 & 1 & 0
	\end{array}\right).$ 
Then it is easy to see that there exists $\alpha=\frac{12}{5},\beta=1$ such that $ 4=\left \| B\right\|_L>\|A\|=\frac{9}{4}$ but 
	\begin{eqnarray*}
	 w(T) &\leq& \frac{1}{\sqrt{\alpha+\beta}}\left \|\frac{\alpha}{4}(|T|+|T^*|)^2 + \beta T^*T \right \|^{\frac{1}{2}}=\sqrt{\frac{32}{17}}\approx 1.37198868114\\
	&<&\frac{1}{2}\|~~ |T|+|T^* |~~\|=\frac{3}{2}=1.5\\
	&<&\frac{1}{2}  \left ( \|T\|+\|T^2\| ^{\frac{1}{2}} \right)=\frac{2+\sqrt{2}}{2} \approx 1.70710678119.
	\end{eqnarray*}
\end{remark}

Next we provide a necessary and sufficient condition for the operator $T$  for which the bounds obtained  in Theorem \ref{est4} (ii) are equal.

\begin{theorem}\label{eql2}
Let $T \in \mathcal{B}(\mathcal{H})$ and let $A=\frac{1}{2}(|T|+|T^*|)$, $B=|T|$, $C=|T|-|T^*|$. If $M_A~~\cap~~ M_B~~\cap~~ \ker C\neq \emptyset$ then $$\frac{1}{\alpha+\beta} \left \| \frac{\alpha}{4} \left ( |T|+|T^*| \right)^2+\beta |T|^2 \right\|=\frac{1}{4}\left\| |T|+|T^*| \right\|^2=\|T\|^2,~~\mbox{for all}~~ \alpha, \beta.$$  Moreover, the converse is also true if $T$ is compact and $\alpha\beta\neq 0$. 
\end{theorem}

\begin{proof}
Let $x\in M_A~~\cap~~ M_B~~\cap~~ \ker C$. Then for all $\alpha,\beta$, we get
\begin{eqnarray*}
\frac{1}{\alpha+\beta}\| \alpha A^2+\beta B^2\|&\geq& \frac{1}{\alpha+\beta}\| (\alpha A^2+\beta B^2)x\|\\
&\geq& \frac{1}{\alpha+\beta}\left\langle (\alpha A^2+\beta B^2)x,x\right\rangle \\
&=& \frac{1}{\alpha+\beta}\left(\alpha \|Ax\|^2+\beta \|Bx\|^2\right)  \\      
&=& \|~~|T|x~~\|^2\\
&=& \|~~|T|~~\|^2\\
&=& \|T\|^2.
\end{eqnarray*}
Also, $\frac{1}{\alpha+\beta}\| \alpha A^2+\beta B^2\|\leq \|T\|^2$ for all $\alpha,\beta.$
Therefore, $\frac{1}{\alpha+\beta}\| \alpha A^2+\beta B^2\|= \|T\|^2$ for all $\alpha,\beta.$ 
Thus, $$\frac{1}{\alpha+\beta} \left \| \frac{\alpha}{4} \left ( |T|+|T^*| \right)^2+\beta |T|^2 \right\|=\frac{1}{4}\left\| |T|+|T^*| \right\|^2=\|T\|^2,~~\mbox{for all}~~ \alpha, \beta.$$
Now we prove the converse part.  Let $T$ be compact and  $\frac{1}{\alpha+\beta} \left \| \frac{\alpha}{4} \left ( |T|+|T^*| \right)^2+\beta |T|^2 \right\|$ $=\frac{1}{4}\left\| |T|+|T^*| \right\|^2=\|T\|^2,~~\mbox{for all}~~ \alpha\beta \neq 0.$ Let $x\in M_{\alpha A^2+\beta B^2}$. Then we have,
\begin{eqnarray*}
\|T\|^2=\frac{1}{\alpha+\beta}\| \alpha A^2+\beta B^2\|&=& \frac{1}{\alpha+\beta}\| (\alpha A^2+\beta B^2)x\|\\
&=& \frac{1}{\alpha+\beta}\langle (\alpha A^2+\beta B^2)x,x\rangle \\
&=& \frac{1}{\alpha+\beta}\left ( \alpha\langle  A^2x,x\rangle+\beta \langle B^2x,x\rangle \right )\\
&= & \frac{\alpha}{\alpha+\beta}\|Ax\|^2+\frac{\beta}{\alpha+\beta}\|Bx\|^2\\
&\leq & \frac{\alpha}{\alpha+\beta}\|A\|^2+\frac{\beta}{\alpha+\beta}\|B\|^2\\
&\leq& \|T\|^2.
\end{eqnarray*}
This implies that $\|Ax\|=\|A\|$ and $\|Bx\|=\|B\|.$ Also from above it follows that  $\|A\|=\|B\|=\|T\|$ and so $Ax=\|T\|x$, $Bx=\|T\|x$. Therefore, $Ax=Bx$, i.e., $x\in \ker C.$ Hence, $x\in M_A~~\cap~~ M_B~~\cap~~ \ker C$.
\end{proof}

\begin{remark} 
$(i) $ If $T\in \mathcal{B}(\mathcal{H})$ is a compact normal operator then it is easy to see  that $M_A~~\cap~~ M_B~~\cap~~ \ker C\neq \emptyset$, ($A,B,C$ are in Theorem \ref{eql2}), and so the inequalities in Theorem \ref{est4} (ii) becomes equalities, i.e.,
$$w^2(T)=\inf_{\alpha,\beta} \frac{1}{\alpha+\beta} \left \| \frac{\alpha}{4} \left ( |T|+|T^*| \right)^2+\beta |T|^2 \right\|=\frac{1}{4}\left\| |T|+|T^*| \right\|^2.$$
$(ii)$  We give an example of a matrix which is not normal but $M_A~~\cap~~ M_B~~\cap~~ \ker C\neq \emptyset,$ ($A,B,C$ are in Theorem \ref{eql2}). Consider $T=\left(\begin{array}{ccc}
0 & 2 & 0\\
0 & 0 & 2\\
0 & 0 & 0
\end{array}\right).$ Then it is easy to see that $\left(\begin{array}{c}
0 \\
1  \\
0 
\end{array}\right) \in M_A~~\cap~~ M_B~~\cap~~ \ker C$, but $2=\|T\| \neq w(T)=\sqrt{2}$, and so  $T$ is not normal.
\end{remark}

We next obtain the following lower bounds and upper bounds for the $ ( \alpha, \beta)$-norm, and apply it to the study of numerical radius inequalities.

\begin{theorem}\label{est5}
Let $T \in \mathcal{B}(\mathcal{H}).$ Then 
\begin{eqnarray*}
 (i)~~  \|T\|_{\alpha,\beta}^2 &\geq& \max \left \{\frac{1}{2} \left \|\frac{\alpha}{4} \left ( T^*T+TT^*   \right)+\beta T^*T \right\|,\frac{1}{3} \left \|\frac{\alpha}{2} \left ( T^*T+TT^* \right)+\beta T^*T \right\| \right \},\\  
 (ii)~~  \|T\|_{\alpha,\beta}^2 &\leq& \left \| \frac{\alpha}{2} \left (  T^*T+TT^*\right) +\beta T^*T \right\|~~\textit{and}\\
(iii)~~ w^2(T) &\leq& \inf_{\alpha,\beta}\frac{1}{\alpha+\beta} \left \| \frac{\alpha}{2} \left ( T^*T+TT^*\right)+\beta T^*T \right\| 
\leq \frac{1}{2} \|T^*T+TT^*\|.
\end{eqnarray*}
\end{theorem}

\begin{proof} $ (i) .$
Let $x$ be a unit vector in $\mathcal{H} $. Then, 
\begin{eqnarray*}
|\langle Tx,x\rangle|^2 &=&  \left ( \langle Re(T)x,x  \rangle ^2+ \langle Im(T)x,x  \rangle^2   \right)  \\
 &\geq& \frac{1}{2}\left( |\langle Re(T)x,x  \rangle| +|\langle Im(T)x,x  \rangle|   \right)^2\\
 &\geq& \frac{1}{2} \left | \langle (Re(T)\pm Im(T) )x,x \rangle \right |^2.
\end{eqnarray*}
Therefore, taking supremum over all unit vectors in $\mathcal{H},$ we get
\begin{eqnarray*}
\sup_{\|x\|=1}\alpha |\langle Tx,x\rangle|^2 &\geq& \frac{\alpha}{2} \sup_{\|x\|=1}\left | \langle (Re(T)\pm Im(T) )x,x \rangle \right |^2\\
&=& \sup_{\|x\|=1} \left \langle \frac{\alpha}{2}(Re(T)\pm Im(T))^2x,x \right \rangle. \\
& \geq & \sup_{\|x\|=1} \left \langle \frac{\alpha}{4} \Big (\left ((Re(T)+ Im(T))^2+((Re(T)- Im(T))^2   \right)\Big)x,x\right \rangle 
\end{eqnarray*}
Also $ \|T\|_{\alpha\,\beta}^2 \geq \sup_{\|x\|=1} \langle \beta T^*Tx,x \rangle .$ 
Therefore, 
\begin{eqnarray*}
2\|T\|_{\alpha\,\beta}^2 &\geq& \sup_{\|x\|=1} \left \langle \left (\frac{\alpha}{4}\left ((Re(T)+ Im(T))^2+((Re(T)- Im(T))^2   \right)+\beta T^*T\right)x,x\right \rangle \\
\Rightarrow \|T\|_{\alpha\,\beta}^2 &\geq& \frac{1}{2} \left \|\frac{\alpha}{2} \left (  (Re(T))^2+ (Im(T))^2  \right)+\beta T^*T \right\|\\
 \Rightarrow \|T\|_{\alpha\,\beta}^2 &\geq&  \frac{1}{2} \left \|\frac{\alpha}{4} \left ( T^*T+TT^*   \right)+\beta T^*T \right\|.............................(1)
\end{eqnarray*}

 Also, we have
\begin{eqnarray*}
3\|T\|_{\alpha\,\beta}^2 &\geq& \sup_{\|x\|=1} \left \langle \left (\frac{\alpha}{2}\left ((Re(T)+ Im(T))^2+((Re(T)- Im(T))^2   \right)+\beta T^*T\right)x,x\right \rangle \\
\Rightarrow \|T\|_{\alpha\,\beta}^2 &\geq& \frac{1}{3} \left \|\alpha \left (  (Re(T))^2+ (Im(T))^2  \right)+\beta T^*T \right\|\\
 \Rightarrow \|T\|_{\alpha\,\beta}^2 &\geq&  \frac{1}{3} \left \|\frac{\alpha}{2} \left ( T^*T+TT^* \right)+\beta T^*T \right\|.............................(2)
\end{eqnarray*}
Combining $(1) $ and $(2)$, we get the  inequality $(i)$.  Next we prove the second inequality. We have,
\begin{eqnarray*}
\alpha |\langle Tx,x\rangle|^2 +\beta \|Tx\|^2 &=& \alpha \left ( \langle Re(T)x,x  \rangle ^2+ \langle Im(T)x,x  \rangle^2   \right) +\beta \langle T^*Tx,x \rangle \\
&\leq& \alpha \left (\|Re(T)x\|^2 +\|Im(T)x\|^2 \right )+\beta \langle T^*Tx,x \rangle \\
&=& \alpha \left (\left \langle (Re(T))^2x,x   \right \rangle+\left \langle  (Im(T))^2x,x \right \rangle \right)+\beta \langle T^*Tx,x \rangle \\
&=& \left \langle \left (\alpha \left ( (Re(T))^2+(Im(T))^2 \right) +\beta T^*T  \right)x,x  \right \rangle \\
&=&  \left \langle \left (\frac{\alpha}{2}  ( T^*T+TT^*) +  \beta T^*T \right)x,x  \right \rangle. 
\end{eqnarray*}
Therefore, taking supremum over all unit vectors in $\mathcal{H},$ we get the inequality (ii). It follows from the inequality in (ii), by using Theorem \ref{norm} that 
\begin{eqnarray*}
w^2(T) &\leq& \frac{1}{\alpha+\beta} \left \| \frac{\alpha}{2} \left ( T^*T+TT^*\right)+\beta T^*T \right\|.
 \end{eqnarray*}
Taking infimum over all $\alpha,\beta$ we get the inequality
$$ w^2(T) \leq \inf_{\alpha,\beta}\frac{1}{\alpha+\beta} \left \| \frac{\alpha}{2} \left ( T^*T+TT^*\right)+\beta T^*T \right\|.$$ 
The remaining inequality follows from the case $ \alpha =1, \beta = 0 .$ 
\end{proof}

We would like to note that the first inequality in Theorem \ref{est5}(iii) improves on the existing inequality in \cite[Th. 1]{K05}, namely,  $$w^2(T)\leq \frac{1}{2} \|T^*T+TT^*\|.$$
Following the ideas involved in Theorem \ref{less2} and Theorem \ref{eql2}, we can prove the following two theorems which provide conditions under which our bound is strictly sharper or equal to the existing bound.

\begin{theorem}\label{less1}
	Let $T\in \mathcal{B}(\mathcal{H}).$   Suppose there exists $\alpha,\beta$ with $\alpha\beta\neq 0$ such that  $ M_{\alpha A+\beta B}\neq \emptyset,$ where $A=\frac{1}{2}(T^*T+TT^*)$ and $B=T^*T.$  If $\left \| B\right\|_L<\|A\|$, where $L=\mbox{span} \{M_{\alpha A+\beta B}\}$,   then $$w^2(T) \leq \frac{1}{\alpha +\beta} \left \| \frac{\alpha}{2} \left ( T^*T+TT^*\right)+\beta T^*T \right\|<\frac{1}{2}\| T^*T+TT^* \|.$$
\end{theorem}

\begin{theorem}\label{eql1}
Let $T \in \mathcal{B}(\mathcal{H})$ and let $A=\frac{1}{2}(T^*T+TT^*)$, $B=T^*T$, $C=T^*T-TT^*$. If $M_A~~\cap~~ M_B~~\cap~~ \ker C\neq \emptyset$ then $$\frac{1}{\alpha+\beta} \left \| \frac{\alpha}{2} \left ( T^*T+TT^*\right)+\beta T^*T \right\|=\frac{1}{2}\| T^*T+TT^* \|=\|T\|^2,~~\mbox{for all}~~ \alpha, \beta.$$  Moreover, the converse is also true if $T$ is compact and $\alpha\beta\neq 0$. 
\end{theorem}

\begin{remark} 
	
$(i)$	 We give an example to show that the condition mentioned in Theorem \ref{less1} is not necessary.  Consider $T=\left(\begin{array}{ccc}
	0 & 0 & 0\\
	2 & 0 & 0\\
	0 & 1 & 0
	\end{array}\right).$ 
	Then it is easy to see that there exists $\alpha=6,\beta=1$ such that $ 4=\left \| B\right\|_L>\|A\|=\frac{5}{2}$ but 
	$$ w^2(T) \leq \frac{1}{\alpha+\beta} \left \| \frac{\alpha}{2} \left ( T^*T+TT^*\right)+\beta T^*T \right\|={\frac{16}{7}} \approx 2.2857<\frac{1}{2}\| T^*T+TT^* \|=2.5.$$

$(ii) $ If $T\in \mathcal{B}(\mathcal{H})$  is normaloid operator and $\mathcal{H}$ is finite-dimensional then it is easy to see  that $M_A~~\cap~~ M_B~~\cap~~ \ker C\neq \emptyset$ and so the inequalities in Theorem \ref{est5} (iii) becomes equalities, i.e.,
$$ w^2(T) = \inf_{\alpha,\beta}\frac{1}{\alpha+\beta} \left \| \frac{\alpha}{2} \left ( T^*T+TT^*\right)+\beta T^*T \right\| 
= \frac{1}{2} \|T^*T+TT^*\|.$$

$(iii)$  We give an example of a matrix which is not normaloid but $M_A~~\cap~~ M_B~~\cap~~ \ker C\neq \emptyset.$ Consider $T=\left(\begin{array}{ccc}
0 & 2 & 0\\
0 & 0 & 2\\
0 & 0 & 0
\end{array}\right).$ Then it is easy to see that $\left(\begin{array}{c}
0 \\
1  \\
0 
\end{array}\right) \in M_A~~\cap~~ M_B~~\cap~~ \ker C$, but $2=\|T\| \neq w(T)=\sqrt{2}$, i.e., $T$ is not normaloid. 
\end{remark}

To prove the next result, we need the Buzano's inequality (see \cite{B}), which is a generalization of the Cauchy-Schwarz inequality. 

\begin{lemma}\label{lemma1}
Let $a,b,e\in \mathcal{H}$ with $\|e\|=1$. Then 
\[|\langle a,e\rangle ~\langle e,b\rangle|\leq \frac{1}{2}\left(\|a\|~\|b\|+|\langle a,b\rangle|\right).\]
\end{lemma}

\begin{theorem}\label{theorem2}
Let $T\in \mathcal{B}({\mathcal{H}}).$ Then 
\begin{eqnarray*}
(i)~~\|T\|^2_{\alpha,\beta}&\leq& \frac{\alpha}{2}w(T^2)+\left \| \frac{\alpha}{4}(T^*T+TT^*)+\beta T^*T \right \|~~\textit{and}\\
(ii)~~w^2(T)  &\leq & \inf_{\alpha,\beta} \frac{1}{\alpha+\beta}\left \{ \frac{\alpha}{2}w(T^2)+\left \|\frac{\alpha}{4}(T^*T+ TT^*)+\beta T^*T \right \| \right\}\\
    &\leq& \frac{1}{2}w(T^2)+\frac{1}{4}\left \|T^*T+ TT^*\right \|.
\end{eqnarray*}
\end{theorem}

\begin{proof}
Let $x\in \mathcal{H}$ with $\|x\|=1.$ Taking $a=Tx, b=T^*x$ and $e=x$ in Lemma \ref{lemma1} and using the arithmetic-geometric mean inequality,  we get,
\begin{eqnarray*}
|\langle Tx,x\rangle|^2&\leq& \frac{1}{2}(\|Tx\|~\|T^*x\|+|\langle T^2x,x\rangle|)\\
&=& \frac{1}{2}|\langle T^2x,x\rangle|+\frac{1}{2}\langle T^*Tx,x\rangle^{1/2}\langle TT^*x,x\rangle^{1/2}\\
&\leq& \frac{1}{2}|\langle T^2x,x\rangle|+\frac{1}{4}(\langle T^*Tx,x\rangle+\langle TT^*x,x\rangle)\\
&=& \frac{1}{2}|\langle T^2x,x\rangle|+\frac{1}{4} \langle (T^*T+ TT^*)x,x\rangle.
\end{eqnarray*}
Therefore, 
\begin{eqnarray*}
\alpha |\langle Tx,x\rangle|^2+\beta \|Tx\|^2 &\leq& \frac{\alpha}{2}|\langle T^2x,x\rangle|+ \left \langle \left (\frac{\alpha}{4}(T^*T+ TT^*)+\beta T^*T \right )x,x\right \rangle.\\
&\leq& \frac{\alpha}{2}w(T^2)+\left \|\frac{\alpha}{4}(T^*T+ TT^*)+\beta T^*T \right \|.
\end{eqnarray*}
Taking supremum over all unit vectors in $\mathcal{H}$, we get the desired inequality in (i). Using  Theorem \ref{norm} and taking infimum over $\alpha, \beta, $ we get the inequality  
$$ w^2(T)  \leq  \inf_{\alpha,\beta} \frac{1}{\alpha+\beta}\Big \{ \frac{\alpha}{2}w(T^2)+\left \|\frac{\alpha}{4}(T^*T+ TT^*)+\beta T^*T \right \| \Big\} .$$
The remaining inequality follows from the case $ \alpha =1, \beta = 0 .$ 
\end{proof}

We next state the  following theorem which  provide a condition on the operator $T$ for which the bound obtained by us  in  Theorem \ref{theorem2} (ii)  is sharper than the existing bound $ \frac{1}{2}w(T^2)+\frac{1}{4}\left \|T^*T+ TT^*\right \|.$ The proof is omitted as it follows using previous arguments.

\begin{theorem}\label{less3}
Let $T\in \mathcal{B}(\mathcal{H}).$ Suppose there exists $\alpha,\beta$ with $\alpha\beta\neq 0$ such that  $ M_{\alpha A+\beta B}\neq \emptyset,$ where $A=\frac{1}{4}(T^*T+TT^*)$ and $B=T^*T.$  If $\left \| B\right\|_L<\|A\|$, where $L=\mbox{span} \{M_{\alpha A+\beta B}\}$,   then 
$$ w^2(T)\leq \frac{1}{\alpha+\beta}\left \{ \frac{\alpha}{2}w(T^2)+\left \|\frac{\alpha}{4}(T^*T+ TT^*)+\beta T^*T \right \| \right\} < \frac{1}{2}w(T^2)+\frac{1}{4}\left \|T^*T+ TT^*\right \|.$$
\end{theorem}

\begin{remark}
We give an example to show that the condition mentioned in Theorem \ref{less3} is not necessary.  Consider $T=\left(\begin{array}{ccc}
	0 & 0 & 0\\
	2 & 0 & 0\\
	0 & 1 & 0
	\end{array}\right).$ 
	 Then it is easy to see that there exists $\alpha=12,\beta=1$ such that $ 4=\left \| B\right\|_L>\|A\|=\frac{5}{4}$ but
	\begin{eqnarray*}
	w^2(T) &\leq& \frac{1}{\alpha+\beta}\left \{ \frac{\alpha}{2}w(T^2)+\left \|\frac{\alpha}{4}(T^*T+ TT^*)+\beta T^*T \right \| \right\}=\frac{16.5}{13} \approx 1.26923077 \\
	&<&\frac{1}{2}w(T^2)+\frac{1}{4}\left \|T^*T+ TT^*\right \|=\frac{7}{4}=1.75.
\end{eqnarray*}
\end{remark}

Our penultimate result  is on the estimation of the  upper bound for the $(\alpha,\beta)$-norm of bounded linear operator $T$  in terms of  $ Re(T), Im(T)$. For this  we first need the following lemma.

\begin{lemma}\cite{K88}\label{lem9}
Let $T \in \mathcal{B}(\mathcal{H})$ be self-adjoint and let $x\in \mathcal{H}$.Then
\[|\langle Tx,x \rangle | \leq \langle |T|x,x \rangle .\]
\end{lemma}

\begin{theorem}\label{est5a}
Let $T \in \mathcal{B}(\mathcal{H}).$ Then
\begin{eqnarray*}
(i)~~\|T\|^2_{\alpha,\beta} &\leq& \left\| \alpha(|Re(T)|+|Im(T)|)^2 +\beta T^*T \right\|~~\textit{and}\\
(ii)~~w^2(T)  &\leq&  \inf_{\alpha,\beta} \frac{1}{\alpha+\beta} \left\| \alpha(|Re(T)|+|Im(T)|)^2 +\beta T^*T \right\|.
\end{eqnarray*}
\end{theorem}

\begin{proof}
Let $x$ be a unit vector in $\mathcal{H} $. Then by using Lemma \ref{lem9} and Lemma \ref{lem5}, we get,
\begin{eqnarray*}
\alpha |\langle Tx,x \rangle|^2+\beta \|Tx\|^2 &=& \alpha \left | \langle Re(T)x,x \rangle +i \langle Im(T)x,x \rangle \right|^2 +\beta \langle T^*Tx,x \rangle \\
&\leq& \alpha \left ( |\langle Re(T)x,x \rangle| +|\langle Im(T)x,x \rangle | \right)^2 +\beta \langle T^*Tx,x \rangle \\
&\leq& \alpha \left ( \langle |Re(T)|x,x \rangle +\langle |Im(T)|x,x \rangle  \right)^2 +\beta \langle T^*Tx,x \rangle \\
&=& \alpha \left ( \langle (|Re(T)|+|Im(T)|)x,x \rangle \right)^2 +\beta \langle T^*Tx,x \rangle \\
&\leq&    \left\langle \left ( \alpha\left (|Re(T)|+|Im(T)|\right)^2+\beta T^*T\right)x,x \right \rangle  \\
&\leq& \left \| \alpha \left(|Re(T)|+|Im(T)|\right)^2+\beta T^*T \right \|.
\end{eqnarray*}
Therefore, taking supremum over all unit vectors in $\mathcal{H},$ we get the required inequality in (i). Using  Theorem \ref{norm} in (i) and taking infimum over $\alpha, \beta, $ we get the inequality in (ii), i.e.,  
$$ w^2(T)  \leq  \inf_{\alpha,\beta} \frac{1}{\alpha+\beta} \left\| \alpha(|Re(T)|+|Im(T)|)^2 +\beta T^*T \right\|.$$
\end{proof}

\begin{remark}  In particular, if we take $\alpha = 1,\beta = 0$ in Theorem \ref{est5a} (ii), then we get the following inequality 
	$$w(T) \leq \Big \| |Re(T)|+|Im(T)| \Big \|.$$
	The existing upper bound of numerical radius in  terms of $ Re(T)  $ and $ Im(T)$ is 
	$$w^2(T) \leq \|Re(T)\|^2+\|Im(T)\|^2.$$ We give an example to show that the upper bound obtained here is better than the existing one. 
	 Considering  the matrix $T=\left(\begin{array}{cc}
	1 & 0\\
	0 & i
	\end{array}\right),$  it is easy to see  that the inequality $w^2(T) \leq \|Re(T)\|^2+\|Im(T)\|^2$ gives $w(T)\leq \sqrt{2}$, whereas our inequality  gives $w(T) \leq 1.$ 
\end{remark} 

Before presenting the final result of this article, we would like to note that the introduction of the $ (\alpha, \beta)$-norm on $ \mathcal{B}(\mathcal{H}) $ has played a crucial role in improving the $\emph{upper bounds} $ of  the existing numerical radius inequalities. Therefore, a natural query in this context would be regarding the possible improvements of the $\emph{lower bounds} $ of  the existing numerical radius inequalities, using similar techniques. Our final result is dedicated to answering this query in a fruitful way. Given an arbitrary but fixed $ T \in \mathcal{B}(\mathcal{H}), $ let us consider an expression $f_{\alpha, \beta}(x)=\alpha|\langle Tx,x\rangle|^2+\frac{\beta}{4}\|(T^*T+TT^*)x\|$, where $x\in \mathcal{H}.$ It is quite straightforward to check that the above expression does not induce a norm on $\mathcal{B}(\mathcal{H}) $ by taking the supremum over all unit vectors. Nevertheless, we show in the following theorem that it is possible to  improve the existing lower bound for the numerical radius of bounded linear operators, obtained in \cite[Th. 1]{K05}, namely,  $w^2(T)\geq \frac{1}{4} \|T^*T+TT^*\|,$ by using the above expression.

\begin{theorem}\label{lower}
Let $T\in \mathcal{B}(\mathcal{H})$ and let $\mathcal{H}_0=\{x\in \mathcal{H}: \|(T^*T+TT^*)x\|=\|T^*T+TT^*\|\|x\|\}.$ Let $\mathcal{H}_0\neq \emptyset$ and let $T|_{\mathcal{H}_0}=T_0.$ Then 
\begin{eqnarray*}
w^2(T)&\geq &\sup_{\alpha,\beta} \left\{ \frac{\alpha}{\alpha+\beta} \sup_{\|x\|=1} | \langle T_0x,x\rangle |^2+\frac{\beta}{4(\alpha+\beta)}\left \|T^*T+TT^* \right \| \right \}\\
&\geq& \frac{1}{4} \left \|T^*T+TT^* \right \|.
\end{eqnarray*}
\end{theorem}

\begin{proof}
We have, 
\begin{eqnarray*}
\sup_{\|x\|=1}f_{\alpha, \beta}(x)&=&\sup_{\|x\|=1}\{\alpha|\langle Tx,x\rangle|^2+\frac{\beta}{4}\|(T^*T+TT^*)x\|\}  \\
&\leq& \alpha w^2(T)+\frac{\beta}{4}\|T^*T+TT^*\|\\
&\leq& (\alpha+\beta) w^2(T).
\end{eqnarray*}
Also, $$\sup_{\|x\|=1,x\in \mathcal{H}_0 }f_{\alpha, \beta}(x)=\alpha \sup_{\|x\|=1} | \langle T_0x,x\rangle |^2+\frac{\beta}{4}\|T^*T+TT^*\|.$$
Therefore, we get,  $$ w^2(T) \geq \frac{1}{\alpha+\beta}\left \{ \alpha \sup_{\|x\|=1} | \langle T_0x,x\rangle |^2+\frac{\beta}{4}\|T^*T+TT^*\| \right \}.$$
Since this holds for all admissible values of $\alpha,\beta$, the first inequality follows. The second inequality follows trivially by considering  the particular case $\alpha=0,\beta=1$.
 \end{proof}

We end this article with the following remark that justifies that Theorem \ref{lower} is a proper refinement of the existing lower bound for the numerical radius as obtained in \cite[Th. 1]{K05}, for a large class of operators.
 
\begin{remark}
$(i)$ We note that, if $\mathcal{H}_0$ is invariant under $T,$  then $\sup_{\|x\|=1}|\langle T_0x,x\rangle|^2= w^2(T_0)\geq \frac{1}{4} \left \|T^*T+TT^* \right \|.$ Therefore, for all $\alpha,\beta,$ 
$$\frac{\alpha}{\alpha+\beta}w^2(T_0)+\frac{\beta}{4(\alpha+\beta)}\left \|T^*T+TT^* \right \|\geq \frac{1}{4} \left \|T^*T+TT^* \right \|.$$ 
Moreover, if $w^2(T_0)> \frac{1}{4} \left \|T^*T+TT^* \right \|$, then the first inequality in Theorem \ref{lower} is strictly sharper than the first inequality in \cite[Th. 1]{K05}, obtained by Kittaneh. \\
$(ii) $ Even if $ \mathcal{H}_0$ is not invariant under $T$ then also the lower bound obtained here may give a better bound than that in \cite[Th. 1]{K05}. As an illustrative example, let us consider $T=\left(\begin{array}{ccc}
	1 & 1 & 0\\
	0 & 0 & 0\\
	0 & 0 & 1
	\end{array}\right).$ Then it is easy to check that $\mathcal{H}_0$ is not invariant under $T$ but  $\sup_{\|x\|=1}|\langle T_0x,x\rangle|^2=\left(\frac{4+3\sqrt{2}}{4+2\sqrt{2}}\right)^2 > \frac{1}{4}\|T^*T + TT^*\| = \frac{2 + \sqrt{2} }{4}$ so that for all $\alpha, \beta $ with $ \alpha \neq 0$ we get,
	$$ w^2(T) \geq \left(\frac{\alpha}{\alpha + \beta}\right) \left(\frac{4+3\sqrt{2}}{4+2\sqrt{2}}\right)^2 + \left (\frac{\beta}{\alpha + \beta} \right)\frac{2+\sqrt{2}}{4} > \frac{2 + \sqrt{2} }{4}.$$ 
	Therefore, for this matrix $T$, the first inequality in Theorem \ref{lower} gives a better bound than that in \cite[Th. 1]{K05}.

\noindent $(iii)$ We also note that if 	$\sup_{\|x\|=1} | \langle T_0x,x\rangle |^2= \frac{1}{4} \left \|T^*T+TT^* \right \|, $ then for all $\alpha,\beta$, 
$$\frac{\alpha}{\alpha+\beta}\sup_{\|x\|=1} | \langle T_0x,x\rangle |^2+\frac{\beta}{4(\alpha+\beta)}\left \|T^*T+TT^* \right \|= \frac{1}{4} \left \|T^*T+TT^* \right \|,$$ i.e., the first inequality in Theorem \ref{lower} and the first inequality in \cite[Th. 1]{K05} give the same bound.

\end{remark}

\bibliographystyle{amsplain}

\end{document}